\documentclass[11pt,a4paper]{article}

\usepackage[utf8]{inputenc}
\usepackage{bold-extra}
\usepackage{chngcntr}
\usepackage{tikz}
\usepackage{setspace}
\usepackage{url}
\usepackage{amsfonts}
\usepackage{amssymb}
\usepackage{epsfig}
\usepackage{amsmath}
\usepackage{commath}
\usepackage{natbib}
\usepackage{graphicx}
\usepackage{bbm,float}
\usepackage{dsfont}
\usepackage{color}
\usepackage{adjustbox}
\usepackage{amsmath,amssymb,amstext,amsfonts,amsthm,latexsym}
\usepackage{diagbox}
\usepackage{bm}
\usepackage{caption}
\usepackage{subcaption}
\usepackage{multirow}
\usepackage{color}
\usepackage{listings}
\usepackage{titlesec}
\usepackage{enumerate}
\usepackage{environ}
\usepackage{booktabs}
\usepackage{appendix}
\usepackage{bbm}
\usepackage{pgfplots}
\usepackage{makecell}
\numberwithin{equation}{section}
\theoremstyle{plain}

\newtheorem{proposition}{Proposition}[section]

\newtheorem{lemma}{Lemma}[section]

\theoremstyle{definition}

\newtheorem{definition}{Definition}[section]
\newtheorem{algorithm}{Algorithm}[section]
\newtheorem{example}{Example}[section]

\theoremstyle{remark}

\pgfplotsset{compat=1.17}

\begin{document}
\title{\textbf{Modeling coskewness with zero correlation and correlation with zero coskewness}
}
\date{\today}
\author{Carole Bernard\thanks{Carole Bernard, Department of Accounting, Law and Finance, Grenoble Ecole de Management (GEM) and Department of Economics at Vrije Universiteit Brussel (VUB).  (email: \texttt{carole.bernard@grenoble-em.com}).}, Jinghui Chen\thanks{Corresponding author: Jinghui Chen, Department of Economics at Vrije Universiteit Brussel (VUB).  (email: \texttt{jinghui.chen@vub.be}).} \ and Steven Vanduffel\thanks{Steven Vanduffel, Department of Economics at Vrije Universiteit Brussel (VUB). (email: \texttt{steven.vanduffel@vub.be}). }}
\maketitle
\begin{abstract}
This paper shows that one needs to be careful when making statements on potential links between correlation and coskewness. Specifically, we first show that, on the one hand, it is possible to observe any possible values of coskewness among symmetric random variables but zero pairwise correlations of these variables. On the other hand, it is also possible to have zero coskewness and any level of correlation. Second, we generalize this result to the case of arbitrary marginal distributions showing the absence of a general link between rank correlation and standardized rank coskewness.
\end{abstract}


\newpage
\section{Introduction}
Let $X_i\sim F_i$, $i=1,2,\dots,d$ be random variables such that $\mu_i$ and $\sigma_i$ are their respective means and standard deviations, and their second moments are finite. One of the essential characteristics of dependency of a random vector $\bm{X}=(X_1,X_2,\dots,X_d)$ is the k\textsuperscript{th} order standardized central mixed moments
\begin{equation*}
	\mathbb{E}\left(\left(\frac{X_1-\mu_1}{\sigma_1}\right)^{k_1}\left(\frac{X_2-\mu_2}{\sigma_2}\right)^{k_2}\cdots \left(\frac{X_d-\mu_d}{\sigma_d}\right)^{k_d}\right)
\end{equation*}
where $k_i,$ $i=1,2,\dots,d,$ are non-negative integers such that $\sum_{i=1}^{d}k_i=k$. Specifically, the Pearson correlation coefficient \citep{pearson1895vii} is obtained when $k_{1}=k_{2}=1$ ($d=2$) and coskewness is obtained when $k_{1}=k_{2}=k_{3}=1$ ($d=3$).
	
The correlation coefficient between $X_i$ and $X_j$ denoted as $\rho_{ij}$, $i,j=1,2,\dots,d$, is given as 
\begin{equation*}
	\rho_{ij}=\frac{\mathbb{E}((X_i	-\mu_i)(X_j-\mu_j))}{\sigma_{i}\sigma_{j}},
\end{equation*} 
and the correlation matrix is a $d$ by $d$  matrix.
 \cite{jondeau2006optimal} define the $d$ by $d^2$ coskewness matrix of a $d$-dimensional random vector $\bm{X}$, as a matrix that contains all coskewnesses. The coskewness of $X_i$, $X_j$ and $X_k$, denoted by $S(X_i, X_j, X_k)$, $i,j,k= 1,2,\dots,d$, is given as $$S(X_i, X_j,X_k)=\frac{\mathbb{E}((X_i-\mu_{i})(X_j-\mu_{j})(X_k-\mu_{k}))}{\sigma_{i}\sigma_{j}\sigma_{k}}.$$ The coskewness matrix is denoted by $M_d$, so that, for example, when $d=3$,
\begin{equation*}
	M_3 = \left [ \begin{array}{ccc|ccc|ccc}
		s_{111} & s_{112} & s_{113} & s_{211} & s_{212} & s_{213} & s_{311} & s_{312} & s_{313} \\
		s_{121} & s_{122} & s_{123} & s_{221} & s_{222} & s_{223} & s_{321} & s_{322} & s_{323} \\
		s_{131} & s_{132} & s_{133} & s_{231} & s_{232} & s_{233} & s_{331} & s_{332} & s_{333}
	\end{array} \right ],
\end{equation*}
where $s_{ijk}=S(X_i,X_j,X_k)$, $i,j,k = 1,2,3$. The coskewness matrix $M_d$ is invariant w.r.t.\ location and scale parameters, i.e.,\ a linear transformation of $X_i$, $i=1,2,\dots,d$, does not affect $M_d$. However, the coskewness $S(X_i,X_j,X_k)$ generally depends on the marginal distributions and copula among the three variables; see \cite{bernard2023coskewness}.

It is well-known that correlation always takes values in $[-1, 1]$. However, such affirmation is not true for higher co-moments such as coskewness and cokurtosis. In particular, no universal range of values for coskewness works for all distributions; see \cite{bernard2023coskewness}. We thus use the notion of standardized rank coskewness, which is normalized and takes values in $[-1,\ 1]$.

This paper studies whether a relationship exists between correlation and coskewness. At first glance, it is easy to think that the answer is affirmative because the mathematical formulas of correlation and coskewness share some similarities. Moreover, correlations do not determine the dependence but at least impose some structure. For instance, the maximum and minimum correlation between two random variables are obtained by comonotonic and antimonotonic dependence, respectively. Hence, one could expect a link between the second cross and the third cross moment. For example, \cite{beddock2020two} use a split bivariate normal model to illustrate that the coskewness is monotonic to the correlation; see their Table 3. However, such conclusion heavily depends on the model assumed (here the split bivariate normal), and the remaining of this paper is dedicated to showing that, in general, there is no link between correlation and coskewness and that such conclusions can only be made under specific model assumptions.

The paper is organized as follows. In Section~\ref{three norm sec}, we present counterexamples based on three symmetrically distributed random variables. In Section~\ref{three margins}, we generalize the result to the case of random variables with arbitrary marginal distributions. Section~\ref{tail risk} provides some elements to justify statements that appear in previous literature on the link between coskewness and tail risk. The last section draws the conclusion.
	
\section{Correlation and coskewness with symmetric marginals}\label{three norm sec}
In this section, we aim to show that, in general, there is no link between the coskewness and the correlation coefficient in the case of symmetric distributions. Let $F_i$, $i=1,2,3$, be symmetric distributions, i.e., $X_i\sim F_i$. For symmetric case, we have explicit copulas to obtain the maximizing and minimizing coskewness \citep[see][]{bernard2023coskewness}. Moreover, the symmetric distribution appears as a benchmark in many applications in finance, such as optimal portfolio choice. More general distributions are discovered in Section~\ref{three margins}.

The goal of Section~\ref{three norm sec} is to prove the following two propositions.
\begin{proposition}\label{zero corr any cosk} Let $(X_1,X_2, X_3)$ be a random vector with symmetric marginals. For any given value of coskewness, ranging between the minimum and maximum admissible values, there exists a dependence model such that the coskewness among the three variables attains this value, and such that the pairwise correlations are all equal to zero.
\end{proposition}
\begin{proof}
	In Section~\ref{any cosk zero corr}, we construct such a model.
\end{proof}

\begin{proposition}\label{zero cosk any corr} Let $(X_1,X_2, X_3)$ be a random vector with symmetric marginals. For every given set of correlations among the three variables, there exists a dependence model such that their coskewness is equal to zero.
\end{proposition}
\begin{proof}
	In Section~\ref{any corr zero cosk}, we construct such a model.
\end{proof}

\subsection{Arbitrary coskewness and zero correlation}\label{any cosk zero corr}
We recall that the range of possible values for coskewness depends on  the choice of marginal distributions. The following lemma recalls Theorems 3.1 and 3.2 of \cite{bernard2023coskewness}. We thus do not provide a proof. 
\begin{lemma}[Theorems 3.1 and 3.2 of \cite{bernard2023coskewness}] \label{symmetric margins and any copula}
	Let $X_i\sim F_i$ in which the $F_{i}$ are symmetric, $i=1,2,3$, and $U\sim U[0,1]$. The explicit bounds $\underbar{S}$ and $\bar{S}$ for the coskewness of $X_1$, $X_2$ and $X_3$ are 
	\begin{equation*}
\underbar{S}:=		-\mathbb{E}\left(G^{-1}_1(U)G^{-1}_2(U)G^{-1}_3(U)\right)\leq S(X_1, X_2, X_3) \leq \bar{S}:=\mathbb{E}\left(G^{-1}_1(U)G^{-1}_2(U)G^{-1}_3(U)\right),
	\end{equation*}
	in which $G_i$ is the distribution of $\lvert(X_i-\mu_i)/\sigma_i\rvert$. The maximum coskewness $\bar{S}$ is attained for  $\bar{S}=S(Y_1, Y_2, Y_3)$ in which  $Y_i=F_i^{-1}(U_i)$ with $U_i$ as in \begin{equation}\label{max copula for odd dimension}
		\begin{aligned}
			U_1&=U, \\
			U_{2}&=IJU+I(1-J)(1-U)+(1-I)JU+(1-I)(1-J)(1-U), \\
			U_3&=IJU+I(1-J)(1-U)+(1-I)J(1-U)+(1-I)(1-J)U,
		\end{aligned}
	\end{equation}
	where $I=\mathds{1}_{U>\frac{1}{2}}$, $J=\mathds{1}_{V>\frac{1}{2}}$, and $V\overset{d}{=}U[0, 1]$ is independent of $U$. The minimum coskewness
	$\underbar{S}$ is attained for  $\underbar{S}=S(H_1, H_2, H_3)$ in which  $H_i=F_i^{-1}(U_i)$ with $U_i$ as in \begin{equation}\label{min copula for odd dimension}
		\begin{aligned}
			U_1&=U, \\
			U_{2}&=IJU+I(1-J)(1-U)+(1-I)JU+(1-I)(1-J)(1-U), \\
			U_3&=IJ(1-U)+I(1-J)U+(1-I)JU+(1-I)(1-J)(1-U).
		\end{aligned}
	\end{equation} 
\end{lemma}
When $F_{i}$, $i=1,2,3$, are symmetric, the bounds can be computed explicitly; see Table~2 in \cite{bernard2023coskewness}.

We now construct a model in which the coskewness varies from $\underbar{S}$ to $\bar{S}$ but where the pairwise correlations of these variables are always equal to zero. To do so, let us introduce a mixture copula $C^{\lambda}$ for $\lambda\in[0,1]$ based on Lemma~\ref{symmetric margins and any copula}. We refer to \cite{lindsay1995mixture} for a study on mixture models. 

\begin{definition}[Mixture Copula]\label{mix copula defi}
	Let $X_i\sim F_i$, $i=1,2, 3$, in which the $F_i$ are symmetric, $U\overset{d}{=}V\sim U[0,1]$ such that $U\perp V$, $B\sim Bernoulli(\lambda)$ where $B$ is independent of $X_i$, $U$ and $V$, and $\lambda\in [0,1]$. Define two indicator functions $I=\mathds{1}_{U>\frac{1}{2}}$ and $J=\mathds{1}_{V>\frac{1}{2}}$. The dependence structure of $X_1= F_1^{-1}(U_1)$, $X_2= F_2^{-1}(U_2)$ and $X_3= F_3^{-1}(U_3^\lambda)$ is called a mixture copula $C^{\lambda}$ when the trivariate random vector $(U_1, U_2, U_3^\lambda)$ is given as
	\begin{equation}\label{mixture copula}
		\begin{aligned}
			U_1&=U,\\
			U_{2}&=IJU+I(1-J)(1-U)+(1-I)JU+(1-I)(1-J)(1-U), \\
			U_3^{\lambda}&=BU_3^M+(1-B)U_3^m,
		\end{aligned}
	\end{equation} where 
	\begin{align*}
		U_3^M &=IJU+I(1-J)(1-U)+(1-I)J(1-U)+(1-I)(1-J)U
	\end{align*}
	and
	\begin{align*}
		U_3^m &=IJ(1-U)+I(1-J)U+(1-I)JU+(1-I)(1-J)(1-U).
	\end{align*}
\end{definition} 
$U_j$ in \eqref{max copula for odd dimension} and $U_j$ in \eqref{min copula for odd dimension} are the same for $j=1,2$, thus $U_j^m=U_j^M=U_j^{\lambda}=U_j$. Note that \cite{mcneil2022attainability} use the same principle to mix $U$ and $1-U$ to study the property of Kendall's tau. 
\begin{proposition}\label{how mix work}
	Let $(X_1,X_2,X_3)$ be a trivariate random vector with symmetric marginals $F_i$ for $i=1,2,3$, i.e.\ $X_i\sim F_i$ and having the mixture copula $C^{\lambda}$. The coskewness $S(X_1,X_2,X_3)$ can take any values from the minimum to the maximum by varying the parameter $\lambda$ in the mixture copula $C^{\lambda}$. 
\end{proposition}
\begin{proof}
	Without loss of generality, we assume that $X_i$, $i=1,2,3$, have zero means and unit variances. With the mixture copula $C^{\lambda}$, we have
	\begin{equation*}
		\begin{aligned}
			X_1=&F_1^{-1}(U),\\
			X_2=&F_2^{-1}(U_2), \\
			X_3=&BF_3^{-1}(U_3^M)+(1-B)F_3^{-1}(U_3^m).
		\end{aligned}
	\end{equation*}
	Then, the coskewness of $X_1$, $X_2$ and $X_3$ is
	\begin{equation*}
		\begin{aligned}
			S(X_1, X_2, X_3)=&\mathbb{E}\left(F_1^{-1}(U)F_2^{-1}(U_2)\left(BF_3^{-1}(U_3^M)+(1-B)F_3^{-1}(U_3^m)\right)\right)\\
			=&\mathbb{E}\left(BF_1^{-1}(U)F_2^{-1}(U_2)F_3^{-1}(U_3^M)\right)+\mathbb{E}\left((1-B)F_1^{-1}(U)F_2^{-1}(U_2)F_3^{-1}(U_3^m)\right) \\
			=&\lambda \mathbb{E}\left(F_1^{-1}(U)F_2^{-1}(U_2)F_3^{-1}(U_3^M)\right)+(1-\lambda )\mathbb{E}\left(F_1^{-1}(U)F_2^{-1}(U_2)F_3^{-1}(U_3^m)\right)\\
			=&\lambda \bar{S}+(1-\lambda )\underbar{S}.
		\end{aligned}
	\end{equation*} The third equation holds because $B$ is independent of $X_i$, $U$ and $V$.
\end{proof}
The proof of Proposition~\ref{how mix work} shows that the mixture copula $C^\lambda$ leads to a coskewness that is a linear combination between the maximum coskewness and the minimum coskewness with weights driven by the parameter $\lambda$.  We then consider a trivariate random vector with symmetric marginals and the mixture copula $C^{\lambda}$. Thus, the mixture random variables $X_j=F_{j}^{-1}(U_j)$ for $j=1,2$ and $X_3^{\lambda}=F_{3}^{-1}(U_3^{\lambda})$, in which $U_j$ and $U_3^\lambda$ are in \eqref{mixture copula}. Let us denote this model as $(X_1, X_2, X_3^{\lambda})$. In Appendix~\ref{sim mix copula}, we provide a numerical method to simulate the dependence structure $C^{\lambda}$ and thus the model $(X_1, X_2, X_3^{\lambda})$.

\begin{figure}[!h]
	\centering
	\includegraphics[width=1\textwidth]{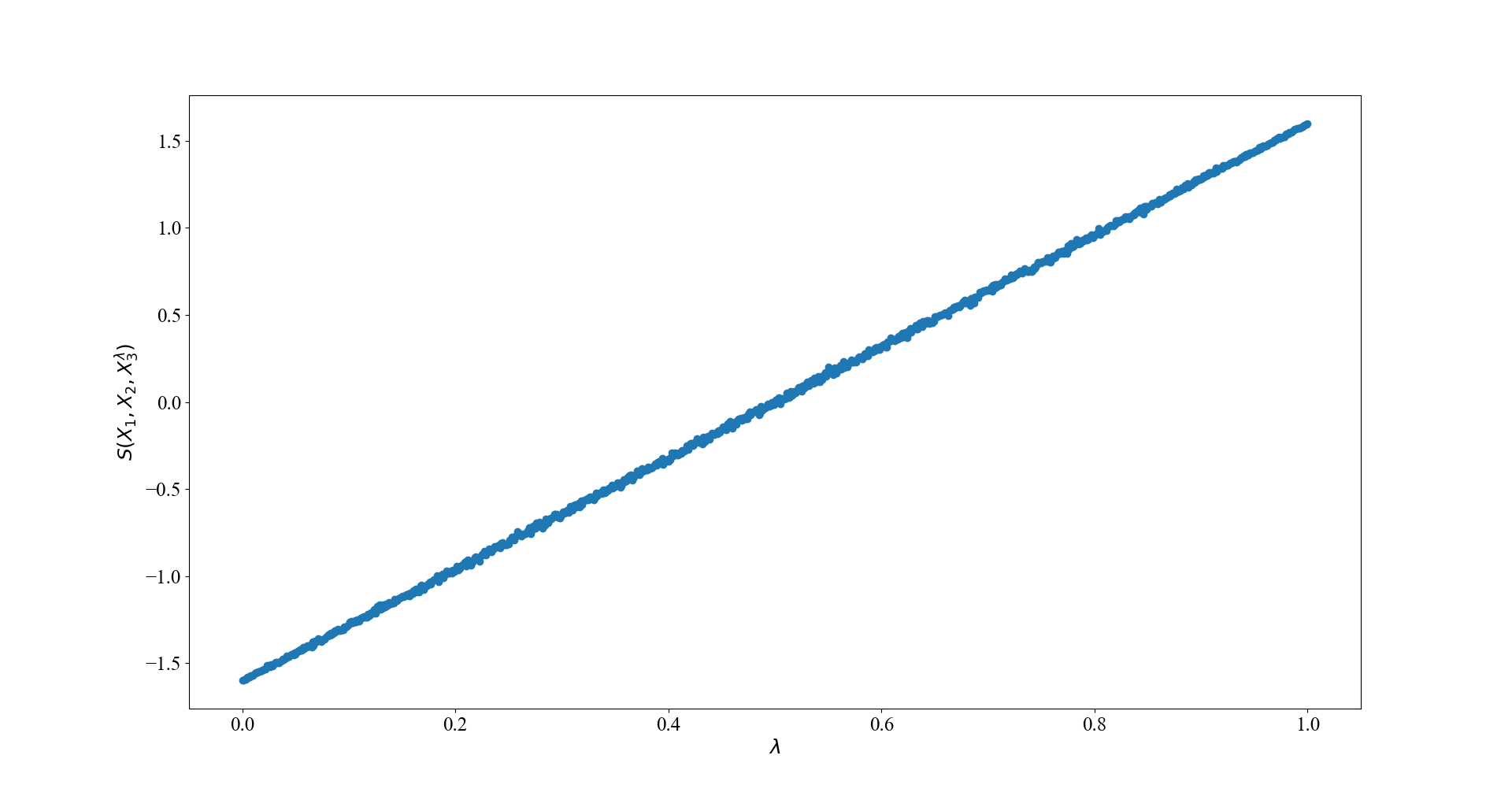}
	\caption{Effect of the parameter $\lambda$ on coskewness in the case of three normal variables. The coskewness is obtained by simulation with number of simulations $n=10^5$. The approximate minimum ($\lambda=0$) and maximum coskewness ($\lambda=1$) are $-1.59\approx -\frac{2\sqrt{2\pi}}{\pi}$ and $1.59\approx \frac{2\sqrt{2\pi}}{\pi}$, respectively.}
	\label{cosk vary}
\end{figure}


We proceed by simulating the mixture copula $C^{\lambda}=(U_1, U_2, U_3^{\lambda})$ using Algorithm~\ref{mix alg dim3} in Appendix A. Figure~\ref{cosk vary} illustrates that this model allows us to span all possible levels of coskewness. This result follows immediately by the construction of the mixture copula and by the continuity of coskewness with respect to the parameter $\lambda$. Moreover, the plot proves Proposition~\ref{how mix work} numerically. Given the behaviour of coskewness as a linear function of $\lambda$, we can then use $\lambda$ to represent the level of coskewness.

We can prove that the correlation coefficient is equal to zero in this mixture model. Hence, we obtain the following proposition.
\begin{proposition}\label{corr 3d prop}
	Let $(X_1, X_2, X_3)$ be a trivariate random vector with symmetric marginals $F_i$ for $i=1,2,3$, i.e.\ $X_i\sim F_i$ and having the mixture copula $C^{\lambda}$. The pairwise correlation coefficients of the three variables are equal to zero, while their coskewness takes arbitrary values $($depending on the value of $\lambda$$)$ between the minimum and the maximum.
\end{proposition}
\begin{proof}
	We only need to prove that correlations are equal to zero. Without loss of generality, we assume that all $X_i$ are symmetrically distributed random variables with zero means and unit variances. Observe that for an indicator function
	\begin{equation*}
		\mathds{1}_{A}(\omega)=\left\{\begin{aligned}
				&1, \qquad \text{ if } \omega \in A, \\
				&0, \qquad \text{otherwise,}
			\end{aligned}\right.
	\end{equation*}
	 we have $\mathds{1}_{A}f(x)+(1-\mathds{1}_{A})f(y)=f(\mathds{1}_{A}x+(1-\mathds{1}_{A})y)$ for all functions. Note that $I$, $J$, $B$, $1-I$, $1-J$ and $1-B$ in dependence structure $C^{\lambda}$ are all indicator functions as well as their products. Thus, under assumptions of symmetric marginals and dependence structure $C^{\lambda}$, we have
	\begin{equation*}
		\begin{aligned}
			X_1=&F_1^{-1}(U),\\
			X_2=&IJF_2^{-1}(U)+I(1-J)F_2^{-1}(1-U)+(1-I)JF_2^{-1}(U)+(1-I)(1-J)F_2^{-1}(1-U), \\
			X_3=&IJF_3^{-1}(BU+(1-B)(1-U))+I(1-J)F_3^{-1}(B(1-U)+(1-B)U)\\
			&+(1-I)JF_3^{-1}(B(1-U)+(1-B)U)+(1-I)(1-J)F_3^{-1}(BU+(1-B)(1-U)).
		\end{aligned}
	\end{equation*}
	Note that $\Phi^{-1}(U)$ in $X_1$ can be expanded as follow
	\begin{equation*}
		X_1=IJF_1^{-1}(U)+I(1-J)F_1^{-1}(U)+(1-I)JF_1^{-1}(U)+(1-I)(1-J)F_1^{-1}(U).
	\end{equation*}
	We now prove that $\rho_{12}$ equals zero using $F_2^{-1}(U)=-F_2^{-1}(1-U)$. We obtain
	\begin{equation*}
		\begin{aligned}
		\rho_{12}=\mathbb{E}(X_1X_2)=&\frac{1}{2}[\mathbb{E}(IF_1^{-1}(U)F_2^{-1}(U))+\mathbb{E}(IF_1^{-1}(U)F_2^{-1}(1-U))\\
		&+\mathbb{E}((1-I)F_1^{-1}(U)F_2^{-1}(U))+\mathbb{E}((1-I)F_1^{-1}(U)F_2^{-1}(1-U))] \\
		=& \frac{1}{2}[\mathbb{E}(IF_1^{-1}(U)F_2^{-1}(U))-\mathbb{E}(IF_1^{-1}(U)F_2^{-1}(U))\\
		&+\mathbb{E}((1-I)F_1^{-1}(U)F_2^{-1}(U))-\mathbb{E}((1-I)F_1^{-1}(U)F_2^{-1}(U))]=0.
	\end{aligned}
	\end{equation*}
	Similarly, we have $\rho_{13}=0$ since $$\mathbb{E}(IF_1^{-1}(U)F_3^{-1}(BU+(1-B)(1-U)))=-\mathbb{E}(IF_1^{-1}(U)F_3^{-1}(B(1-U)+(1-B)U))$$ and
	$$\mathbb{E}((1-I)F_1^{-1}(U)F_3^{-1}(B(1-U)+(1-B)U))=-\mathbb{E}((1-I)F_1^{-1}(U)F_3^{-1}(BU+(1-B)(1-U))).$$ The proof that $\rho_{23}=0$ is similar and omitted.
\end{proof}
Proposition~\ref{zero corr any cosk} follows as a corollary of Proposition~\ref{corr 3d prop}.

\subsection{Arbitrary correlation and zero coskewness}\label{any corr zero cosk}

\begin{proposition}\label{corr 3d prop 2}
	Let $(X_1, X_2, X_3)$ be a trivariate Gaussian random vector. The coskewness $S(X_1, X_2, X_3)$ of $X_1$, $X_2$ and $X_3$, equals zero for any possible values of correlation $\rho_{ij}$ of $X_i$ and $X_j$, $i,j=1,2,3$ and $i\neq j$.
\end{proposition}
\begin{proof}
	We only need to prove that coskewness is equal to zero. It is well-known that the trivariate Gaussian random vector $(X_1, X_2, X_3)$ can be expressed as
	\begin{equation}\label{normal rvs}
		\begin{aligned}
			X_1&=\mu_1+\sigma_1Z_1,\\
			X_2&=\mu_2+\sigma_2\left(\rho_{12}Z_1+aZ_2\right),  \\
			X_3&=\mu_3+\sigma_3\left(\rho_{13}Z_1+\frac{\rho_{23}-\rho_{12}\rho_{13}}{a}Z_2+\frac{b}{a}Z_3\right),
		\end{aligned}
	\end{equation}
	where $a=\sqrt{1-\rho_{12}^2}$, $b=\sqrt{1-\rho_{12}^2-\rho_{13}^2-\rho_{23}^2+2\rho_{12}\rho_{13}\rho_{23}}$, and $Z_1$, $Z_2$ and $Z_3$ are independent standard normally distributed random variables. The coskewness of $X_1$, $X_2$ and $X_3$ is
	\begin{equation*}
		\begin{aligned}
		S(X_1,X_2,X_3)=&\mathbb{E}\left(Z_1\left(\rho_{12}Z_1+aZ_2\right)\left(\rho_{13}Z_1+\frac{\rho_{23}-\rho_{12}\rho_{13}}{a}Z_2+\frac{b}{a}Z_3\right)\right) \\
		=&\rho_{12}\rho_{13}\mathbb{E}Z_1^3+\left(\frac{\rho_{12}(\rho_{23}-\rho_{12}\rho_{13})}{a}+a\rho_{13}\right)\mathbb{E}(Z_1^2Z_2)\\
		&+\frac{b\rho_{12}}{a}\mathbb{E}(Z_1^2Z_3)+(\rho_{23}-\rho_{12}\rho_{13})\mathbb{E}(Z_1Z_2^2)+b\mathbb{E}(Z_1Z_2Z_3)\\
		=&0.
	\end{aligned}
	\end{equation*}
	The last equation for $S(X_1,X_2,X_3)$ holds because $\mathbb{E}Z_1=\mathbb{E}Z_2=\mathbb{E}Z_3=\mathbb{E}Z_1^3=0.$
\end{proof}
Proposition~\ref{zero cosk any corr} follows as a corollary of Proposition~\ref{corr 3d prop 2}.

\section{Rank correlation and rank coskewness}\label{three margins}
The range of possible coskewness generally depends on the choice of distributions. Thus in this section, we study the standardized rank coskewness \citep{bernard2023coskewness} as it always takes values in $[-1,\ 1]$.

We first recall the definitions of the standardized rank coskewness from \cite{bernard2023coskewness} and the rank correlation.
\begin{definition}[Standardized Rank Coskewness]\label{standardized rank coskewness}
	Let $X_i\sim F_{i}$, $i=1,2,3$, such that $F_i$ are strictly increasing and continuous. The standardized rank coskewness of $X_1$, $X_2$ and $X_3$ denoted by $RS(X_1,X_2,X_3)$ is defined as $RS(X_1,X_2,X_3)=\frac{4\sqrt{3}}{9}S(F_{1}(X_1), F_{2}(X_2), F_{3}(X_3))$. Hence,
	\begin{equation*} 
		RS(X_1,X_2,X_3)=32\mathbb{E}\left(\left(F_{1}(X_1)-\frac{1}{2}\right)\left(F_{2}(X_2)-\frac{1}{2}\right)\left(F_{3}(X_3)-\frac{1}{2}\right)\right).
		\label{equation for rank coskewness} 
	\end{equation*}
\end{definition}

\begin{definition}[Rank Correlation]\label{rank correlation}
	Let $X_i\sim F_{i}$, $i=1,2$, such that $F_i$ are strictly increasing and continuous. The Spearman's correlation of $X_1$ and $X_2$ denoted by $\rho_{12}^S$ is defined as
	\begin{equation*} 
		\rho_{12}^S=12\mathbb{E}\left(\left(F_{1}(X_1)-\frac{1}{2}\right)\left(F_{2}(X_2)-\frac{1}{2}\right)\right).
		\label{equation for rank correlation} 
	\end{equation*}
\end{definition}

The goal of Section~\ref{three margins} is to prove the following two propositions.
\begin{proposition}\label{any rs zero rho} Let $X_i\sim F_{i}$, $i=1,2,3$, such that $F_i$ are strictly increasing and continuous. For any value in $[-1,1]$, one can construct a dependence model such that standardized rank coskewness among $X_1$, $X_2$ and $X_3$ has that value. In contrast, the pairwise rank correlations among them are equal to zero.
\end{proposition}
\begin{proof}
	In Section~\ref{any rank cosk zero rank corr}, we construct such a model.
\end{proof}

\begin{proposition}\label{any rho zero rs} Let $X_i\sim F_{i}$, $i=1,2,3$, such that $F_i$ are strictly increasing and continuous. There exists a dependence model such that the pairwise rank correlations among $X_1$, $X_2$ and $X_3$ can take any possible values in $[-1,1]$, but their standardized rank coskewness is zero.
\end{proposition}
\begin{proof}
	In Section~\ref{any rank corr zero rank cosk}, we construct such a model.
\end{proof}

\subsection{Arbitrary rank coskewness and zero rank correlation}\label{any rank cosk zero rank corr}
\begin{proposition}\label{corr 3d prop 3}
	Let $(X_1, X_2, X_3)$ be a trivariate random vector with strictly increasing and continuous marginals $F_i$, $i=1,2,3$, and the mixture copula $C^{\lambda}$. The standardized rank coskewness can take any possible values in $[-1,1]$, while the rank correlation coefficients $\rho_{ij}^S$ of $X_i$ and $X_j$, $j=1,2,3$ and $j\neq i$, are equal to zero.
\end{proposition}
\begin{proof}
	We only need to prove that the rank correlation coefficients are equal to zero. Lemma~\ref{symmetric margins and any copula}, in this case, still holds because $F_i(X_i)$ are distributed as standard uniform. Thus, we have
	\begin{equation*}
		\begin{aligned}
			F_1(X_1)=&U,\\
			F_2(X_2)=&IJU+I(1-J)(1-U)+(1-I)JU+(1-I)(1-J)(1-U), \\
			F_3(X_3)=&IJ(BU+(1-B)(1-U))+I(1-J)(B(1-U)+(1-B)U)\\
			&+(1-I)J(B(1-U)+(1-B)U)+(1-I)(1-J)(BU+(1-B)(1-U)).
		\end{aligned}
	\end{equation*}
	We now prove that rank correlation is equal to zero. It is
	\begin{equation*}
		\begin{aligned}
			\rho_{12}^S=&12\mathbb{E}\left(\left(F_1(X_1)-\frac{1}{2}\right)\left(F_2(X_2)-\frac{1}{2}\right)\right)\\
			=&6\left[\mathbb{E}\left(IU^2\right)+\mathbb{E}\left(IU\left(1-U\right)\right)+\mathbb{E}\left(\left(1-I\right)U^2\right)+\mathbb{E}\left(\left(1-I\right)U\left(1-U\right)\right)-\frac{1}{2}\right] \\
			=&6\left[\mathbb{E}\left(IU\right)+\mathbb{E}\left(\left(1-I\right)U\right)-\frac{1}{2}\right] =0.
		\end{aligned}
	\end{equation*}
	Similarly, we have \begin{equation*}
		\begin{aligned}
			\rho_{13}^S=&12\mathbb{E}\left(\left(F_1(X_1)-\frac{1}{2}\right)\left(F_3(X_3)-\frac{1}{2}\right)\right)\\
			=&6\left[\mathbb{E}\left(I(BU^2+(1-B)(1-U)U)+I(B(1-U)U+(1-B)U^2)\right.\right.\\
			&\left.\left.+(1-I)(B(1-U)U+(1-B)U^2)+(1-I)(BU^2+(1-B)(1-U)U)\right)-\frac{1}{2}\right] \\
			=&6\left[\mathbb{E}\left(BU^2+(1-B)(1-U)U+B(1-U)U+(1-B)U^2\right)-\frac{1}{2}\right] \\
			=&6\left[\mathbb{E}\left((1-B)U+BU\right)-\frac{1}{2}\right]=0.
		\end{aligned}
	\end{equation*} $\rho_{23}^S=0$ can be similarly proven.
\end{proof}
Proposition~\ref{any rs zero rho} follows as a corollary of Proposition~\ref{corr 3d prop 3}.
\subsection{Arbitrary rank correlation and zero rank coskewness}\label{any rank corr zero rank cosk}
\begin{proposition}\label{corr 3d prop 4}
	Let $(X_1, X_2, X_3)$ be a trivariate random vector with strictly increasing and continuous marginals $F_i$, $i=1,2,3,$ and Gaussian copula. The standardized rank coskewness $RS(X_1, X_2, X_3)$ of $X_1$, $X_2$ and $X_3$ is equal to zero for any possible values of rank correlation  $\rho_{ij}^S$ of $X_i$ and $X_j$, $i,j=1,2,3$ and $i\neq j$.
\end{proposition}
\begin{proof} Recall Equation~\eqref{normal rvs}, we have $F_1(X_1)=\Phi(H_1)$, $F_2(X_2)=\Phi(H_2)$ and $F_3(X_3)=\Phi(H_3)$, where
	\begin{equation*}
		\begin{aligned}
			H_1&=Z_1,\\
			H_2&=\rho_{12}Z_1+\sqrt{1-\rho_{12}^2}Z_2,  \\
			H_3&=\rho_{13}Z_1+\frac{\rho_{23}-\rho_{12}\rho_{13}}{\sqrt{1-\rho_{12}^2}}Z_2+\frac{\sqrt{1-\rho_{12}^2-\rho_{13}^2-\rho_{23}^2+2\rho_{12}\rho_{13}\rho_{23}}}{\sqrt{1-\rho_{12}^2}}Z_3.
		\end{aligned}
	\end{equation*} 
	\cite{pearson1907further} proves the relationship between the Pearson correlation and the Spearman rank correlation under Gaussian copula, i.e.\ for $i,j=1,2,3$ and $i\neq j$,
	\begin{equation*}
		\rho_{ij} = 2\sin\left(\frac{\pi}{6}\rho_{ij}^S\right).
	\end{equation*}
	Thus,
	\begin{equation*}
		\rho_{ij}^S = \frac{6}{\pi}\arcsin\left(\frac{\rho_{ij}}{2}\right)\in [-1,1].
	\end{equation*}
	This implies $\mathbb{E}(F_i(X_i)F_j(X_j))=\mathbb{E}(\Phi(H_i)\Phi(H_j))=\frac{1}{2\pi}\arcsin\left(\frac{\rho_{ij}}{2}\right)+\frac{1}{4}$. 
	The rank coskewness of $X_1$, $X_2$ and $X_3$ is
	\begin{equation*}
		\begin{aligned}
			RS(X_1,X_2,X_3)=&32\mathbb{E}\left(\left(F_1(X_1)-\frac{1}{2}\right)\left(F_2(X_2)-\frac{1}{2}\right)\left(F_3(X_3)-\frac{1}{2}\right)\right) \\
			=&32\left[\mathbb{E}(F_1(X_1)F_2(X_2)F_3(X_3))-\frac{1}{2}\mathbb{E}(F_1(X_1)F_2(X_2))-\frac{1}{2}\mathbb{E}(F_1(X_1)F_3(X_3))\right.\\
			&\left.-\frac{1}{2}\mathbb{E}(F_2(X_2)F_3(X_3))+\frac{1}{4}\right]\\
			=&32\left[\mathbb{E}(F_1(X_1)F_2(X_2)F_3(X_3))-\frac{1}{4\pi}\arcsin\left(\frac{\rho_{12}}{2}\right)-\frac{1}{4\pi}\arcsin\left(\frac{\rho_{13}}{2}\right)\right.\\
			&\left.-\frac{1}{4\pi}\arcsin\left(\frac{\rho_{23}}{2}\right)-\frac{1}{8}\right].
		\end{aligned}
	\end{equation*}
	We assume $f_{H_1, H_2, H_3}$ is the joint density function of $(H_1, H_2, H_3)$. Define $X_1'$, $X_2'$ and $X_3'$ as three independent standard normally distributed random variables such that they are independent of $X_1$, $X_2$ and $X_3$. Let $(Y_1,Y_2,Y_3)=\left(\frac{X_1'-H_1}{\sqrt{2}},\frac{X_2'-H_2}{\sqrt{2}},\frac{X_3'-H_3}{\sqrt{2}}\right)$. $(Y_1,Y_2,Y_3)$ is also trivariate normal with zero means and unit variances, and the pairwise correlation coefficients are equal to $\frac{\rho_{ij}}{2}$.
	Then,
	\begin{equation*}
		\begin{aligned}
			\mathbb{E}(F_1(X_1)F_2(X_2)F_3(X_3))=&
			\mathbb{E}(\Phi(H_1)\Phi(H_2)\Phi(H_3)) \\
			=&\int_{\mathbb{R}}\int_{\mathbb{R}}\int_{\mathbb{R}}\Phi(x_1)\Phi(x_2)\Phi(x_3)f_{H_1,H_2,H_3}(x_1,x_2,x_3)dx_1dx_2dx_3\\
			=&\int_{\mathbb{R}}\int_{\mathbb{R}}\int_{\mathbb{R}}\mathbb{P}(X_1'\leq x_1, X_2'\leq x_2, X_3'\leq x_3)f_{H_1,H_2,H_3}(x_1,x_2,x_3)dx_1dx_2dx_3 \\
			=&\int_{\mathbb{R}}\int_{\mathbb{R}}\int_{\mathbb{R}}\mathbb{P}(X_1'\leq x_1, X_2'\leq x_2, X_3'\leq x_3\vert H_1=x_1, H_2=x_2, H_3=x_3)\\
			&f_{H_1,H_2,H_3}(x_1,x_2,x_3)dx_1dx_2dx_3 \\
			=&\mathbb{P}(X_1'\leq H_1, X_2'\leq H_2, X_3'\leq H_3)\\
			=&\mathbb{P}\left(\frac{X_1'-H_1}{\sqrt{2}}\leq 0,\frac{X_2'-H_2}{\sqrt{2}}\leq 0 ,\frac{X_3'-H_3}{\sqrt{2}}\leq 0\right)\\
			=&\mathbb{P}\left(Y_1\leq 0,Y_2\leq 0,Y_3\leq 0\right).
		\end{aligned}
	\end{equation*}
	\cite{rose2002mathematical} prove that 
	\begin{equation*}
		\mathbb{P}\left(Y_1\leq 0,Y_2\leq 0,Y_3\leq 0\right)=\frac{1}{4\pi}\left[\arcsin\left(\rho_{Y_1Y_2}\right)+\arcsin\left(\rho_{Y_1Y_3}\right)+\arcsin\left(\rho_{Y_2Y_3}\right)\right]+\frac{1}{8}.
	\end{equation*}
	Therefore, $RS(X_1,X_2,X_3)=0$.
\end{proof}
Proposition~\ref{any rho zero rs} follows as a corollary of Proposition~\ref{corr 3d prop 4}. Let us illustrate this feature with more examples of rank coskewness in the cases of strictly increasing and continuous marginals and various copulas.
\begin{example}\label{uniform margins example}
	Rank coskewness and rank correlation for strictly increasing and continuous marginals under various dependence assumptions.
	
	\noindent Assume that $X_i\sim F_i$ for $i=1,2,3$ and $U\sim U[0,1]$.
	\begin{enumerate}[(1)]
		\item With the comonotonic copula, we can know $F_{i}(X_i)=U$. The rank coskewness is  $$RS(X_1,X_2,X_3)=32\mathbb{E}\left(\left(U-\frac{1}{2}\right)^3\right)=0.$$
		The rank correlations are $\rho_{12}^S=\rho_{13}^S=\rho_{23}^S=1$.
		\item From \cite{ruschendorf2002variance}, the mixing copula is the dependence such that $\sum_{i=1}^{3}U_i=\frac{3}{2}$ where
		\begin{equation*}
			\begin{aligned}
				U_1&=U; \\
				U_2&=\left\{
				\begin{aligned}
					&-2U+1,  &\text{if } 0\le U \le \frac{1}{2}; \\
					&-2U+2,    &\text{if } \frac{1}{2}\le U \le 1;
				\end{aligned}
				\right. \\
				U_3&=\left\{
				\begin{aligned}
					&U+\frac{1}{2},  &\text{if } 0\le U \le \frac{1}{2}; \\
					&U-\frac{1}{2},    &\text{if } \frac{1}{2}\le U \le 1.
				\end{aligned}
				\right. \\
			\end{aligned}
		\end{equation*}
		Under the mixing copula, we find that the rank coskewness of $X_1$, $X_2$ and $X_3$ is given as
		\begin{equation*}
			\begin{aligned}
				{RS(X_1, X_2, X_3)}
				=&32\left[\mathbb{E}\left(\left(U-\frac{1}{2}\right)\left(-2U+\frac{1}{2}\right)U\mathds{1}_{U\in [0,\frac{1}{2}]}\right)\right. \\
				&\left.+\mathbb{E}\left(\left(U-\frac{1}{2}\right)\left(-2U+\frac{3}{2}\right)\left(U-1\right)\mathds{1}_{U\in [\frac{1}{2},1]}\right)\right] \\
				=&0.
			\end{aligned}
		\end{equation*}
		The rank correlations are $\rho_{12}^S=-1$, $\rho_{13}^S=1$ and $\rho_{23}^S=-1$.
		\item Under the independence copula, we have that $F_i(X_i)=U_i$, in which $X_1$, $X_2$ and $X_3$ are independent. Moreover, the rank coskewness is $$RS(X_1,X_2,X_3)=32\mathbb{E}\left(U_1-\frac{1}{2}\right)\mathbb{E}\left(U_2-\frac{1}{2}\right)\mathbb{E}\left(U_3-\frac{1}{2}\right)=0.$$
		The rank correlations are $\rho_{12}^S=\rho_{13}^S=\rho_{23}^S=0$.
	\end{enumerate}
\end{example}
These examples also support the proof for Proposition~\ref{any rho zero rs}.
\section{Coskewness and tail risk} \label{tail risk}
Rather than using Pearson correlation as in the previous sections, we utilize the event conditional correlation coefficient to analyze the relationship between two random variables, $X_i$ and $X_j$, $i,j=1,2,3$, given a particular event $\mathcal{A}$; see the same definition in \cite{maugis2014event}. This coefficient, denoted as $\rho_{ij|\mathcal{A}}$ and given as 
\begin{equation}\label{con correlation}
	\rho_{ij|\mathcal{A}}=\frac{\mathbb{E}((X_i-\mu_{X_i|\mathcal{A}})(X_j-\mu_{X_j|\mathcal{A}})|\mathcal{A})}{\sigma_{X_i|\mathcal{A}}\sigma_{X_j|\mathcal{A}}}
\end{equation}
quantifies the degree of correlation between $X$ and $Y$, conditioned on event $\mathcal{A}$. Similarly, $\mu_{X_i|\mathcal{A}}$ and $\sigma_{X_i|\mathcal{A}}$ are the respective conditional mean and standard deviation of $X_i$.

One notable application of conditional correlation in risk management is the exceedance correlation, where event $\mathcal{A}$ is defined as exceeding a certain threshold, i.e., $\{X_i>\theta_1, X_j>\theta_2\}$ or $\{X_i\leq \theta_1, X_j\leq \theta_2\}$. \cite{longin2001extreme} first introduce the concept of exceedance correlation to study the dependence structure of international equity markets, while more recent studies, such as \cite{sakurai2020has}, apply this concept to investigate the relationship between oil and the US stock market. In some cases, the exceedance correlation is calculated using the inverse of the cumulative distribution functions of $X_i$ and $X_j$, denoted as $\theta_1=F_i^{-1}(p)$ and $\theta_2=F_j^{-1}(p)$, respectively, in which $p\in [0,1]$. For example, \cite{garcia2011dependence} use this approach to test the co-movement trend between international equity and bond markets. However, it should be noted that the exceedance correlation is constantly equal to one under specific dependence structures, as described by Equations \eqref{max copula for odd dimension} and \eqref{min copula for odd dimension}. Another interesting conditional correlation in finance is when the event $\mathcal{A}$ is the overall volatility of the market ($Z$) greater than a crisis volatility threshold (z), i.e., we consider $\rho_{ij|Z>z_c}$. $X_i$ and $X_j$ can be two asset returns in the conditional correlation. Banks are considerably interested in estimating $\rho_{ij|Z>z_c}$ efficiently. \cite{kalkbrener2015correlation} conducted a study of $\rho_{ij|Z>z_c}$ on determining the appropriate amount of funds to allocate towards crisis management, while \cite{kenett2015partial} researched efficient asset allocation during a crisis.

In this subsection, we investigate the relationship between the coskewness and the downside risk, which is a type of conditional correlation when event $\mathcal{A}$ is $\{S<\mu_S \text{ where } S=\sum_{i=1}^{3}X_i \text{ and } \mu_S=\mathbb{E}S\}$ in \eqref{con correlation}. Downside risk was first proposed by \cite{bawa1977capital} as a measure of risk for developing a capital asset pricing model and has gained significant interest in portfolio optimization. We refer to \cite{lettau2014conditional} and \cite{zhang2021downside} for further applications of downside risk in finance.

\cite{ang2006downside} study the relationship between downside risk and coskewness and find that the risks differ. In this study, we aim to explore if there exists a theoretical connection between downside risk and coskewness risk. To do so, we use the same parameter settings as in Section~\ref{three norm sec} for Algorithm~\ref{mix alg dim3} but adjust the last step to compute the conditional correlation.

Figure~\ref{cosk and downside risk} illustrates the relationship between the coskewness of three normal random variables with the mixture copula $C^{\lambda}$ and the pairwise downside risks, $\rho_{ij|S<\mu_S}$, $i,j=1,2,3$ and $i\neq j$. Our result shows that as the coskewness becomes more negative, the downside risk sharply increases. Moreover, the reduction rate of downside risk slows down as the coskewness increases. Overall, we find that the downside risk decreases as the coskewness increases, confirming the empirical findings of \cite{ang2006downside} and \cite{huang2012extreme}. They conclude that higher downside risk leads to higher average stock returns, while coskewness risk has the opposite effect. That is, higher coskewness is associated with lower downside risk.

\begin{figure}
	\centering
	\includegraphics[width=1\textwidth]{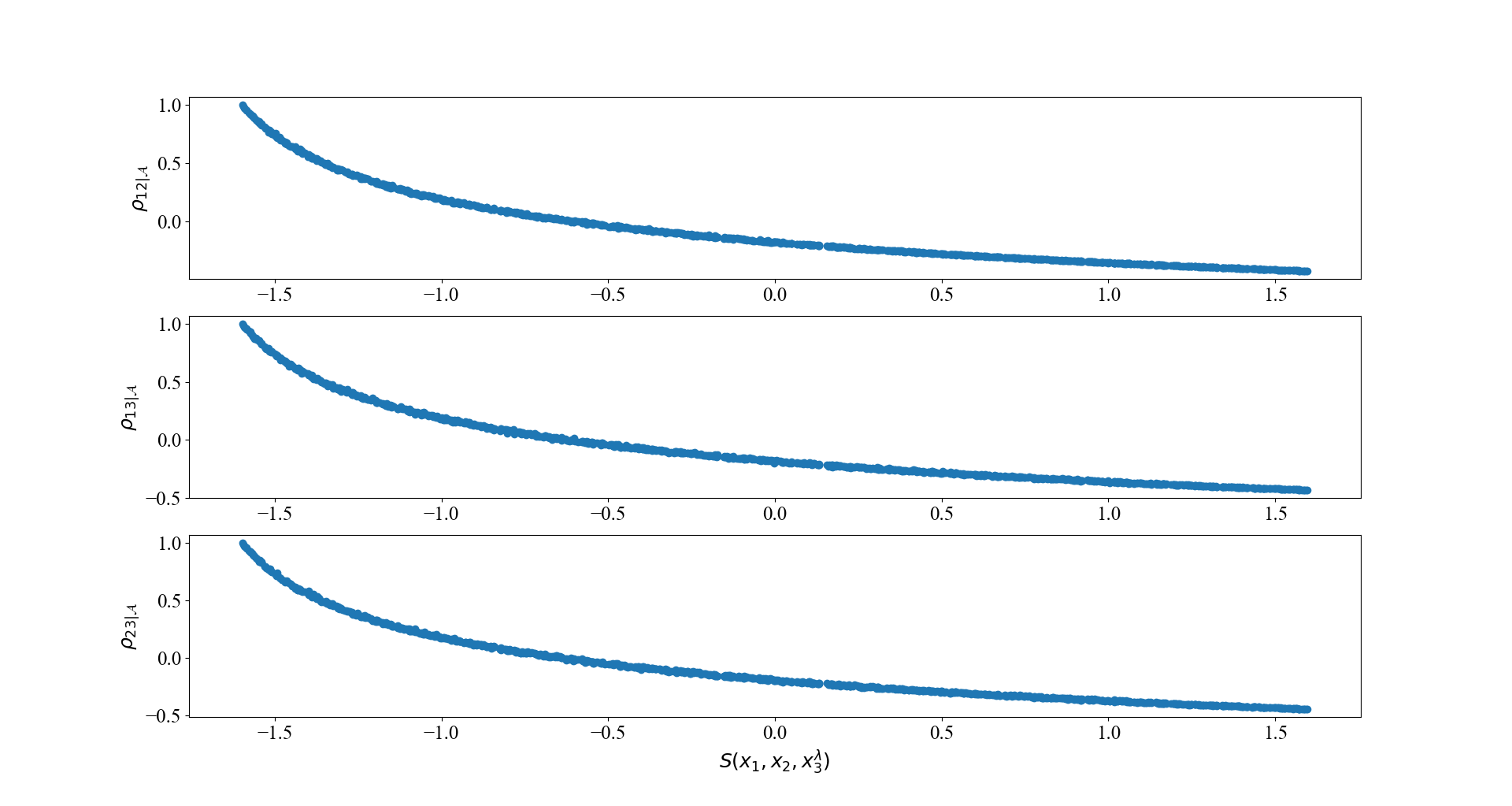}
	\caption{The effect of coskewness on conditional correlations $\rho_{ij|\mathcal{A}}$. The random vector $(X_1,X_2,X_3^{\lambda})$ has normal marginals and the mixture copula $C^\lambda$. The event $\mathcal{A}$ is $\{S<\mu_S \text{ where } S=\sum_{i=1}^{3}X_i \text{ and } \mu_S=\mathbb{E}S\}$. The coskewness and downside risks are obtained by implementing the Algorithm~\ref{mix alg dim3} with $n=10^5$.}
	\label{cosk and downside risk}
\end{figure}


\section{Conclusion}
In this paper, we provide some propositions and examples to illustrate that, in general, there is no link between coskewness and correlation. Under the assumption of some specific models, the coskewness does not affect the correlation, and vice versa. Specifically, the coskewness of three symmetrically distributed random variables takes any values between the maximum and minimum, but the pairwise correlations are equal to zero. Moreover, under the trivariate Gaussian model assumption, the pairwise correlations can reach all possible values, while coskewness equals zero. We generalize the result using the standardized rank coskewness and the rank correlation for all continuous and strictly increasing marginal distributions. Therefore, one needs to be careful when finding potential links between the coskewness and the correlation empirically and theoretically.

\titleformat{\section}[block]{\Large\bfseries}{Appendix \Alph{section}.}{1ex}{}[]
\begin{appendices}
	\section{Simulation of the dependence structures $C^{\lambda}$}\label{sim mix copula}
	As the function $\Phi^{-1}$ is cumbersome to deal with, we propose the following algorithm to compute the coskewness and the pairwise correlation coefficients of mixture random variables. We set $\mu_i=0$ and $\sigma_i=1$ because the location and scale parameters do not affect the coskewness and correlation coefficient.
	\begin{algorithm}\label{mix alg dim3}
		\begin{enumerate}
			\item []
			\item Set the mixture parameter $\lambda\in [0,1]$.
			\item Simulate $\mathbf{u}=(u_1,\dots, u_n)$, $\mathbf{v}=(v_1,\dots,v_n)$ and $\mathbf{b}=(b_1,\dots, b_n)$ where $u_i$, $v_i$ and $b_i$ $i=1,\dots,n$, are respective $n$ sampled values from random variables $U\sim U[0, 1]$, $V\sim U[0, 1]$ and $B\sim Bernoulli(\lambda)$. 
			\item Compute discrete maximizing and minimizing copulas $\mathbf{u_j^M}=(u_{1j}^M,\dots, u_{nj}^M)$ and $\mathbf{u_j^m}=(u_{1j}^m,\dots, u_{nj}^m)$, $j=1,2,3$, using $\mathbf{u}$ and $\mathbf{v}$ in terms of copulas \eqref{max copula for odd dimension} and \eqref{min copula for odd dimension}, respectively.
			\item Compute discrete mixture copula $\mathbf{c_j^{\lambda}}=(c_{1j}^{\lambda},\dots, c_{nj}^{\lambda})$ where $c_{ij}^{\lambda}=b_iu_{ij}^M+(1-b_i)u_{ij}^m$.
			\item Compute discrete mixture random variables $\mathbf{x_j}=(x_{1j},\dots,x_{nj})$ where $x_{ij}=\Phi^{-1}(c_{ij}^{\lambda})$.
			\item Compute $\bar{x}_j=\frac{1}{n}\sum_{i=1}^{n}x_{ij}$ and $s_j=\sqrt{\frac{1}{n}\sum_{i=1}^{n}[(x_{ij}-\bar{x}_j)^2]}$. Then $\rho_{jk}=\frac{\frac{1}{n}\sum_{i=1}^{n}[(x_{ij}-\bar{x}_j)(x_{ik}-\bar{x}_k)]}{s_js_k}$ for $k=1,2,3$ and $k\neq j$, and $S(X_1,X_2,X_3)=\frac{\frac{1}{n}\sum_{i=1}^{n}[(x_{i1}-\bar{x}_1)(x_{i2}-\bar{x}_2)(x_{i3}-\bar{x}_3)]}{s_1s_2s_3}$.
		\end{enumerate}
	\end{algorithm}
\end{appendices}
\newpage
\bibliographystyle{chicago}
\bibliography{Bibliography}

\begin{thebibliography}{}

\bibitem[\protect\citeauthoryear{Ang, Chen, and Xing}{Ang
  et~al.}{2006}]{ang2006downside}
Ang, A., J.~Chen, and Y.~Xing (2006).
\newblock Downside risk.
\newblock {\em Review of Financial Studies\/}~{\em 19\/}(4), 1191--1239.

\bibitem[\protect\citeauthoryear{Bawa and Lindenberg}{Bawa and
  Lindenberg}{1977}]{bawa1977capital}
Bawa, V.~S. and E.~B. Lindenberg (1977).
\newblock Capital market equilibrium in a mean-lower partial moment framework.
\newblock {\em Journal of Financial Economics\/}~{\em 5\/}(2), 189--200.

\bibitem[\protect\citeauthoryear{Beddock and Karehnke}{Beddock and
  Karehnke}{2020}]{beddock2020two}
Beddock, A. and P.~Karehnke (2020).
\newblock Two skewed risks.
\newblock {\em Preprint\/}.

\bibitem[\protect\citeauthoryear{Bernard, Chen, R{\"u}schendorf, and
  Vanduffel}{Bernard et~al.}{2023}]{bernard2023coskewness}
Bernard, C., J.~Chen, L.~R{\"u}schendorf, and S.~Vanduffel (2023).
\newblock Coskewness under dependence uncertainty.
\newblock {\em Statistics \& Probability Letters\/}~{\em 199}, 109853.

\bibitem[\protect\citeauthoryear{Garcia and Tsafack}{Garcia and
  Tsafack}{2011}]{garcia2011dependence}
Garcia, R. and G.~Tsafack (2011).
\newblock Dependence structure and extreme comovements in international equity
  and bond markets.
\newblock {\em Journal of Banking \& Finance\/}~{\em 35\/}(8), 1954--1970.

\bibitem[\protect\citeauthoryear{Huang, Liu, Rhee, and Wu}{Huang
  et~al.}{2012}]{huang2012extreme}
Huang, W., Q.~Liu, S.~G. Rhee, and F.~Wu (2012).
\newblock Extreme downside risk and expected stock returns.
\newblock {\em Journal of Banking \& Finance\/}~{\em 36\/}(5), 1492--1502.

\bibitem[\protect\citeauthoryear{Jondeau and Rockinger}{Jondeau and
  Rockinger}{2006}]{jondeau2006optimal}
Jondeau, E. and M.~Rockinger (2006).
\newblock Optimal portfolio allocation under higher moments.
\newblock {\em European Financial Management\/}~{\em 12\/}(1), 29--55.

\bibitem[\protect\citeauthoryear{Kalkbrener and Packham}{Kalkbrener and
  Packham}{2015}]{kalkbrener2015correlation}
Kalkbrener, M. and N.~Packham (2015).
\newblock Correlation under stress in normal variance mixture models.
\newblock {\em Mathematical Finance\/}~{\em 25\/}(2), 426--456.

\bibitem[\protect\citeauthoryear{Kenett, Huang, Vodenska, Havlin, and
  Stanley}{Kenett et~al.}{2015}]{kenett2015partial}
Kenett, D.~Y., X.~Huang, I.~Vodenska, S.~Havlin, and H.~E. Stanley (2015).
\newblock Partial correlation analysis: Applications for financial markets.
\newblock {\em Quantitative Finance\/}~{\em 15\/}(4), 569--578.

\bibitem[\protect\citeauthoryear{Lettau, Maggiori, and Weber}{Lettau
  et~al.}{2014}]{lettau2014conditional}
Lettau, M., M.~Maggiori, and M.~Weber (2014).
\newblock Conditional risk premia in currency markets and other asset classes.
\newblock {\em Journal of Financial Economics\/}~{\em 114\/}(2), 197--225.

\bibitem[\protect\citeauthoryear{Lindsay}{Lindsay}{1995}]{lindsay1995mixture}
Lindsay, B.~G. (1995).
\newblock {\em Mixture models: theory, geometry, and applications}.
\newblock Institute of Mathematical Statistics.

\bibitem[\protect\citeauthoryear{Longin and Solnik}{Longin and
  Solnik}{2001}]{longin2001extreme}
Longin, F. and B.~Solnik (2001).
\newblock Extreme correlation of international equity markets.
\newblock {\em Journal of Finance\/}~{\em 56\/}(2), 649--676.

\bibitem[\protect\citeauthoryear{Maugis}{Maugis}{2014}]{maugis2014event}
Maugis, P. (2014).
\newblock Event conditional correlation: Or how non-linear linear dependence
  can be.
\newblock {\em arXiv preprint arXiv:1401.1130\/}.

\bibitem[\protect\citeauthoryear{McNeil, Ne{\v{s}}lehov{\'a}, and Smith}{McNeil
  et~al.}{2022}]{mcneil2022attainability}
McNeil, A.~J., J.~G. Ne{\v{s}}lehov{\'a}, and A.~D. Smith (2022).
\newblock On attainability of {K}endall’s tau matrices and concordance
  signatures.
\newblock {\em Journal of Multivariate Analysis\/}~{\em 191}, 105033.

\bibitem[\protect\citeauthoryear{Pearson}{Pearson}{1895}]{pearson1895vii}
Pearson, K. (1895).
\newblock Note on regression and inheritance in the case of two parents.
\newblock {\em Proceedings of the Royal Society of London\/}~{\em
  58\/}(347-352), 240--242.

\bibitem[\protect\citeauthoryear{Pearson}{Pearson}{1907}]{pearson1907further}
Pearson, K. (1907).
\newblock {\em On further methods of determining correlation}, Volume~16.
\newblock Dulau and Company.

\bibitem[\protect\citeauthoryear{Rose, Smith, et~al.}{Rose
  et~al.}{2002}]{rose2002mathematical}
Rose, C., M.~D. Smith, et~al. (2002).
\newblock {\em Mathematical statistics with Mathematica}, Volume~1.
\newblock Springer.

\bibitem[\protect\citeauthoryear{R{\"u}schendorf and Uckelmann}{R{\"u}schendorf
  and Uckelmann}{2002}]{ruschendorf2002variance}
R{\"u}schendorf, L. and L.~Uckelmann (2002).
\newblock Variance minimization and random variables with constant sum.
\newblock In {\em Distributions with given marginals and statistical
  modelling}, pp.\  211--222. Springer.

\bibitem[\protect\citeauthoryear{Sakurai and Kurosaki}{Sakurai and
  Kurosaki}{2020}]{sakurai2020has}
Sakurai, Y. and T.~Kurosaki (2020).
\newblock How has the relationship between oil and the us stock market changed
  after the {C}ovid-19 crisis?
\newblock {\em Finance Research Letters\/}~{\em 37}, 101773.

\bibitem[\protect\citeauthoryear{Zhang, Li, Xiong, and Wang}{Zhang
  et~al.}{2021}]{zhang2021downside}
Zhang, W., Y.~Li, X.~Xiong, and P.~Wang (2021).
\newblock Downside risk and the cross-section of cryptocurrency returns.
\newblock {\em Journal of Banking \& Finance\/}~{\em 133}, 106246.

\end{thebibliography}
\end{document}